\documentclass[12pt, leqno, a4paper]{amsart}
\usepackage{amsmath,mathrsfs}
\usepackage{amssymb}
\usepackage{color}
\usepackage{calligra}
\usepackage{dsfont}
\usepackage{pstricks,pst-node}
\usepackage[latin1]{inputenc}
\usepackage{todonotes}

\usepackage{pdfsync}

\newcommand\NoBlackBoxes{\global\overfullrule0pt}
\NoBlackBoxes

\usepackage{enumitem}
\setitemize{noitemsep,topsep=0pt,parsep=0pt,partopsep=0pt}

\newcommand{\eps}{\varepsilon}

\newcommand{\todistr}{\overset{d}{\underset{N\to\infty}\longrightarrow}}

\newcommand{\N}{\mathbb{N}}

\newcommand{\Z}{\mathbb{Z}}

\renewcommand{\P}{\mathbb{P}}

\newcommand{\Cov}{\mathop{\mathrm{Cov}}\nolimits}

\newcommand{\eee}{{\rm e}}
\newcommand{\dd}{{\rm d}}

\makeatletter
\let\serieslogo@\relax
\let\@setcopyright\relax

\theoremstyle{plain}
\newtheorem{theorem}{Theorem}[section]
\newtheorem{lemma}[theorem]{Lemma}
\newtheorem{corollary}[theorem]{Corollary}

\theoremstyle{definition}

\theoremstyle{remark}
\newtheorem{remark}[theorem]{Remark}

\renewcommand{\P}{{\mathbb{P}}}
\newcommand{\E}{{\mathbb{E}}}
\newcommand{\R}{{\mathbb{R}}}

\newcommand{\C}{\mathbb{C}}

\newcommand{\V}{\mathbb{V}}
\newcommand{\Co}{\mathrm{Cov}}

\renewcommand{\epsilon}{\varepsilon}
\renewcommand{\phi}{\varphi}

\numberwithin{equation}{section}

\begin{document}

\setcounter{page}{1}

\title[Fluctuations for Ising models on dense Erd\H{o}s-R\'enyi graphs]{Fluctuations of the Magnetization for Ising models on dense Erd\H{o}s-R\'enyi random graphs}

\author[Zakhar Kabluchko]{Zakhar Kabluchko}
\address[Zakhar Kabluchko]{Fachbereich Mathematik und Informatik,
Universit\"at M\"unster,
Einsteinstra\ss e 62,
48149 M\"unster,
Germany}

\email[Zakhar Kabluchko]{zakhar.kabluchko@uni-muenster.de}

\author[Matthias L\"owe]{Matthias L\"owe}
\address[Matthias L\"owe]{Fachbereich Mathematik und Informatik,
Universit\"at M\"unster,
Einsteinstra\ss e 62,
48149 M\"unster,
Germany}

\email[Matthias L\"owe]{maloewe@math.uni-muenster.de}

\author[Kristina Schubert]{Kristina Schubert}
\address[Kristina Schubert]{ Fakult\"at f\"ur Mathematik, TU Dortmund, Vogelpothsweg 87, 44227 Dortmund,
Germany}

\email[Kristina Schubert]{kristina.schubert@tu-dortmund.de}


\date{\today}



\newcommand{\wlim}{\mathop{\hbox{\rm w-lim}}}
\newcommand{\na}{{\mathbb N}}
\newcommand{\re}{{\mathbb R}}

\newcommand{\vep}{\varepsilon}

\begin{abstract}
We analyze Ising/Curie-Weiss models on the (directed) Erd\H{o}s-R\'enyi random graph on $N$ vertices in which every edge is present with probability $p$.
These models were introduced by Bovier and Gayrard [\emph{J.\ Stat.\ Phys.}, 1993]. We prove a quenched Central Limit Theorem for the magnetization in the high-temperature regime $\beta<1$ when $p=p(N)$ satisfies $p^3N^2\to +\infty$.
\end{abstract}

\maketitle

\section{Introduction and main results}
\subsection{Description of the model}
The topic of this note are Ising models on random graphs, more precisely the Erd\H{o}s-R\'enyi random graph. These models of disordered ferromagnets were introduced in the physics literature (see \cite{froehlichlecture} for a classic survey). First rigorous results go back to Georgii \cite{georgii_dilute}, the model we are analysing in this note was introduced and rigorously studied in \cite{BG93b}.

To define this model, let $G=G(N,p)$ be a realization of a directed Erd\H{o}s-R\'enyi graph with loops, so for all $i,j \in \{1, \ldots, N\}$ (which may coincide) the directed edge $(i,j)$ is present with probability $p\in (0,1]$, independent of all other edges. For the general model, we assume that $p=p(N)$ depends on $N$ in such a way that $pN \to \infty$ as $N\to\infty$ to ensure that there is a giant component comprising almost all of the vertices. However, as it turns out, the model seems to have a change of behavior in $p$ when $p^3N^2 $ is of constant order. Hence in this note we focus on the case $p^3N^2 \to \infty$ and treat the other cases in later works.

We denote by $\vep_{i,j}$ the indicator variable which equals $1$ if the edge $(i,j)$ is present in the graph. That is, $(\vep_{i,j})_{i,j=1}^{N}$ are independent random variables with
$$
\P[\vep_{i,j} = 1] = p, \qquad \P[\vep_{i,j} = 0] = 1- p.
$$
The Hamiltonian of the Ising model on $G$ is a function $H:= H_N: \{-1,+1\}^N \to\R$, which can be written as
\begin{equation}\label{hamilCW}
H(\sigma) = - \frac 1 {2Np} \sum_{i,j=1}^N \vep_{i,j} \sigma_i \sigma_j
\end{equation}
for $\sigma = (\sigma_1,\ldots,\sigma_N) \in \{-1,+1\}^N$. The associated Gibbs measure is a random probability measure on $\{-1,+1\}^N$ given by
\begin{equation}\label{gibbs}
\mu_\beta (\sigma):=\frac 1 {Z_{N}(\beta)} \exp(-\beta H(\sigma)), \qquad \sigma \in \{-1,+1\}^N,
\end{equation}
where $\beta \ge 0$ is called the inverse temperature, while the quantity
\begin{equation}\label{partition}
Z_N(\beta):= \sum_{\sigma \in \{-1,+1\}^N} \exp(-\beta H(\sigma))
\end{equation}
is called the partition function. It encodes much of the interesting information of the system.
The limit
\begin{equation}\label{free_energy}
-\lim_{N \to \infty} \frac {1} {N\beta} \log Z_{N}(\beta),
\end{equation}
if it exists, is called the free energy per site or particle.
The normalization in \eqref{hamilCW} has been chosen in such a way that the critical temperature of the model is $\beta_c=1$. To see this, define the magnetization (per particle) of the (dilute) Curie-Weiss-Ising model to be
$$
m_N(\sigma)= \frac{\sum_{i=1}^N \sigma_i}{N}.
$$
Since often we will simply use $\sum_{i=1}^N \sigma_i$, let us put
\begin{equation}
|\sigma|:= N m_N(\sigma)= \sum_{i=1}^N \sigma_i.
\end{equation}
In the standard Curie-Weiss model, i.e.\ in the case $p=1$, this quantity was studied extensively. A key tool for the investigation are large deviation techniques, see e.g.\ \cite{Ellis_Newman_78a}, \cite{Ellis_Newman_78b}, \cite{EiseleEllis-MultiplePhaseTransitionsInGCW}, or the monograph \cite{Ellis-EntropyLargeDeviationsAndStatisticalMechanics}. The main finding is that the model exhibits a phase transition at $\beta =1$. While in the high temperature regime $\beta \le 1$ the magnetization $m_N$ converges to 0 as the system size $N$ goes to infinity, it is concentrated around two values, $m^+$ and $-m^+$ for some strictly positive $m^+$, if $\beta >1$ (the low temperature regime). Also the free energy per site is a vanishing quantity at high temperatures while it is not vanishing in the low temperature regime.
As was shown in \cite{BG93b}, the same holds true for dilute Curie-Weiss-Ising models, if $p N \to \infty$: For $\beta \le 1$ the magnetization $m_N$ converges to $0$ under the Gibbs measure for almost all realizations of the random graph, while this is not the case for larger $\beta$. In this latter case the distribution of $m_N$ under the Gibbs measure has an almost sure limit with respect to the realization of the random graph given by
$$
\frac 12 (\delta_{m^+(\beta)}+\delta_{-m^+(\beta)}),
$$
where $\delta_x$ is the Dirac-measure in a point $x$ and $m^+(\beta)$ is the largest solution of
$$
z= \tanh(\beta z).$$
Also note that for convenience and consistency with \cite{BG93b} we consider directed Erd\H{o}s-R\'enyi graphs. However, the results for undirected ones should agree with our findings.

The present paper is inspired by two results: The first of them concerns the fluctuations of $m_N$ in the Curie-Weiss model. In \cite{Ellis_Newman_78b}, \cite{Ellis-EntropyLargeDeviationsAndStatisticalMechanics}, \cite{EL10}, \cite{Chatterjee_Shao} it was shown that in the Curie-Weiss model,  $\sqrt N m_N$ converges in distribution to a centered normal random variable with variance $\frac 1{1-\beta}$ when $\beta <1$, while for $\beta =1$ one has to scale differently. Here one obtains that $\sqrt[4] N m_N$ converges in distribution to a non-normal random variable
with Lebesgue density proportional to $\exp(-\frac 1 {12} x^4)$. The second motivation for the present note were results on the thermodynamics of Ising models on random graphs.
While Ising models  have been studied on different random graph models, most of the models share a locally tree-like, i.e.~sparse, random graph structure.
 The thermodynamic quantities in such models were analyzed i.e.~by Dembo and Montanari in \cite{Dembo_Montanari_2010a} and \cite{Dembo_Montanari_2010b} as well as Giardina and van der Hofstad with coauthors in \cite{van_der_Hofstad_et_al_2010},
\cite{van_der_Hofstad_et_al_2014}, \cite{van_der_Hofstad_et_al_2015}, \cite{van_der_Hofstad_et_al_2015b}, and \cite{van_der_Hofstad_et_al_2015c}. However, the first rigorous result on dilute Ising or Curie-Weiss models is probably due to Bovier and Gayrard \cite{BG93b} and the setting in \cite{BG93b} is of a different nature than in the other references. While many ideas in \cite{Dembo_Montanari_2010a}, \cite{Dembo_Montanari_2010b}, \cite{van_der_Hofstad_et_al_2010},
\cite{van_der_Hofstad_et_al_2014}, \cite{van_der_Hofstad_et_al_2015b}, \cite{van_der_Hofstad_et_al_2015c}, and \cite{van_der_Hofstad_et_al_2015} exploit the almost tree-like structure of the underlying graph (i.e.~in a neighborhood of a given vertex one hardly finds any circles),  the authors in \cite{BG93b} compare the model to that of the fully connected graph (which stands no chance of being successful on sparse graphs).
Their main idea in the proof is to show that for $Np \to \infty$ with overwhelming probability for all $\sigma$ the set of aligned pairs of spins (i.e.~$\sigma_i \sigma_j=1$) that are connected by an edge
has a size that is close to its expected size.
This admits estimates that show that a large deviations principle for the magnetization carries over from the Curie-Weiss model to the Ising model on almost all the realizations of the Erd\H{o}s-R\'enyi graph. This precise form of the argument was exploited in \cite{LSV19}, \cite{LSblock1}, while this step was performed differently in \cite{BG93b}.  On the other hand, the estimates in \cite{BG93b} do not allow to also transfer the above stated fluctuation results for the magnetization to Curie-Weiss models on these random graphs, because the difference between the quenched energy and the expected energy is of a too large order (roughly of order at least $\sqrt{N/p}$). To analyze them is the aim of the present note.

\subsection{Main results}
Our first result concerns the distribution of the normalized magnetization $\sqrt N m_N(\sigma) = |\sigma|/\sqrt N$  under the Gibbs measure $\mu_\beta$ defined in~\eqref{gibbs}.
Denote by $\mathcal M(\R)$ the space of probability measures on $\R$ endowed with the topology of weak convergence. It is well-known that weak convergence can be metrized by the L\'evy metric
$$
d_{L}(\mu_1,\mu_2) = \inf\{\eps>0\colon \mu_1(-\infty,t-\eps]-\eps \leq \mu_2(-\infty,t] \leq \mu_1(-\infty,t+\eps]+\eps\},
$$
which turns $\mathcal M(\R)$ into a complete, separable metric space. Consider the following random element of $\mathcal M(\R)$:
\begin{equation}\label{eq:def_L_N}
L_N := \frac 1{Z_N(\beta)} \sum_{\sigma\in \{-1,+1\}^N} \eee^{-\beta H(\sigma)}  \delta_{\frac 1 {\sqrt N} \sum_{i=1}^N \sigma_i},
\end{equation}
where again $\delta_x$ denotes the Dirac measure in $x$.
Note that $L_N$ is a \textit{random} element of $\mathcal \R$ because it depends on the random variables $(\eps_{i,j})_{i,j=1}^N$ generating the random graph. Let also $\mathfrak N_{0,\sigma^2} \in  \mathcal M(\R)$ be the normal probability distribution with mean $0$ and variance $\sigma^2$.

\begin{theorem}\label{theo:magnetization}
Assume that $0<\beta<1$ and let $p=p(N)$ be such that $p^3 N^2 \to\infty$ as $N\to\infty$.
Then, $L_N$, considered as a random element of $\mathcal M(\R)$, converges in probability to $\mathfrak N_{0, 1/(1-\beta)}$.
That is to say, for every $\eps>0$,
$$
\lim_{N\to\infty} \P[d_{L}(L_N, \mathfrak N_{0, 1/(1-\beta)})>\eps] = 0.
$$
\end{theorem}

A key tool for the proof of Theorem \ref{theo:magnetization} is the quantity $Z_N(\beta, g)$ defined below, whose analysis may be interesting in its own right. To define it,
denote by $\mathcal C_b(\R)$ the space of bounded, continuous functions $g:\R\to\R$.
For any $g \in \mathcal{C}_b(\R)$ consider the following generalization of the partition function:
\begin{equation}\label{ZN(g)}
Z_N(\beta, g):= \sum_{\sigma \in \{-1, +1\}^N} \eee^{-\beta H(\sigma)} g\left( \frac{\sum_{i=1}^N \sigma_i}{\sqrt N} \right).
\end{equation}
Note that $Z_N(\beta)= Z_N(\beta, 1)$ is the partition function defined in~\eqref{partition} and
\begin{equation}
\E_{\mu_\beta }\left[ g\left( \frac{\sum_{i=1}^N \sigma_i}{\sqrt N} \right)\right]= \frac{Z_N(\beta, g)}{Z_N(\beta)},
\end{equation}
where, for a fixed disorder $(\eps_{i,j})_{i,j=1}^N$,  $\E_{\mu_\beta}$ denotes the expectation with respect to the Gibbs measure $\mu_\beta$.
We will prove the following

\begin{theorem}\label{theoCW1}
Fix  $\beta\in (0,1)$ and let $p=p(N)$ be such that $p^3N^{2} \to \infty$ as $N\to\infty$.
Then, for all non-negative $g \in \mathcal{C}_b(\R)$, $g\not \equiv 0$,
\begin{equation}
\frac{Z_N(\beta, g)}{\E Z_N(\beta,g)} \to 1
\end{equation}
in $L^2$ and, hence, in probability.
Here, $\E$ denotes expectation with respect to the probability measure $\P$, i.e.~the randomness generated by $(\eps_{i,j})_{i,j=1}^N$.
\end{theorem}

\begin{remark}
We restrict ourselves to the consideration of non-negative $g \in \mathcal{C}_b(\R)$, $g\not \equiv 0$ in Theorem \ref{theoCW1} to avoid a separate consideration of those cases where $\E Z_N(\beta,g)=0$. For our purposes, this will be sufficient, however, a more general statement would, in principle, be possible.
\end{remark}

On our way to prove Theorem \ref{theoCW1} we will analyze the expectation and the covariances of $Z_N(\beta,g)$ and we find:

\begin{theorem}\label{EZNg}
Fix $\beta \in (0,1)$.
For any non-negative $g \in \mathcal{C}_b(\R)$, $g\not \equiv 0$, the expected value of the (generalized) partition function $\E Z_N (\beta,g)$ has the following asymptotics, where $\xi$ is a standard normally distributed random variable.
\begin{enumerate}
\item[(a)]
If $pN \to \infty$, we have
\begin{equation}\label{eq EZNg}
\E Z_N (\beta,g)
\sim
\exp\left( \frac{-\beta^2}{8}+ N^2p \left(\cosh\left(\frac{\beta}{2Np}\right)-1\right)\right)2^N
\E_\xi [g(\xi) \eee^{\frac \beta 2 \xi^2}].
\end{equation}
\item[(b)]
 For $p^3N^2 \to \infty$, \eqref{eq EZNg} boils down to
\begin{equation}\label{eq EZNg:2}
\E Z_N (\beta,g)\sim \eee^{\frac{(1-p)\beta^2}{8p}}2^N\E_\xi [g(\xi) \eee^{\frac \beta 2 \xi^2}].
\end{equation}
\item[(c)]
If $g \equiv 1$ and $p^3N^2 \to \infty$, we have  for the partition function
\begin{equation}\label{eq EZNg:3}
\E Z_N (\beta)\sim \eee^{\frac{(1-p)\beta^2}{8p}} \frac {2^N} {\sqrt{1-\beta}}.
\end{equation}
\end{enumerate}
\end{theorem}
\begin{remark}
Here,  for two sequences $(a_N)_{N\in\N}$ and $(b_N)_{N\in\N}$ we write $a_N \sim b_N$, if their quotient converges to $1$, as $N\to\infty$.
\end{remark}

The next result shows that in the regime $p^3 N^2\to\infty$ the expectation of $Z_N(\beta, g)$ has a larger order of magnitude than its standard deviation.

\begin{theorem}\label{VarZ2}
Fix $\beta\in (0,1)$ and assume that $p=p(N)$ is such that $p^3 N^2\to \infty$ as $N\to\infty$.  For any non-negative $g \in \mathcal{C}_b(\R)$, $g\not \equiv 0$, we have
\begin{equation}\label{V EZNg2}
\lim_{N\to\infty} \frac{\V(Z_N(\beta,g))}{(\E Z_N(\beta,g))^2} =0.
\end{equation}
\end{theorem}
The article is organized in the following way. In the next section we will state and prove some technical results in which we prepare the proofs of our central statements. In Section 3 we will prove some auxiliary results on the expectation and the variance of $\mathbb E[\eee^{-\beta H(\sigma)}]$. Finally, in Section 4, we prove Theorems \ref{theoCW1}, \ref{EZNg}, and  \ref{VarZ2}. The latter also yields the proof of Theorem \ref{theo:magnetization}.

\section{Technical preparation}
In this section we will prepare for the proof of  Theorems \ref{theoCW1}, \ref{EZNg}, and \ref{VarZ2}. The reader may skip this technical section and return to it when necessary.
We will frequently encounter the following function:
$$
F(p, z) := \log (1 - p + p\eee^{z}).
$$
For the purpose of the remainder of this section, $p$ and $z$ are arbitrary complex variables. In particular, $p$ is not required to denote a probability. We will need the power series expansion of $F(p,z)$ around the point $(0,0)$.
For $|p|< 2$ and $|z|< z_0$ with sufficiently small $z_0>0$, we have $|p\eee^z-p| < 1$. Thus, $F(p,z)$ is an analytic function of two complex variables $p$ and $z$ on the domain
$$
\mathcal D = \{(p,z)\in\C^2\colon |p| < 2, |z|<z_0\}.
$$
As such, it has a power series expansion which converges uniformly and absolutely on compact subsets of this domain. Note that by absolute convergence, we can re-arrange and re-group the terms in an arbitrary way.
The first few terms of the power series expansion are
$$
F(p, z) = p z + \frac{p (1-p)}2 z^2+ \frac{ p(2p^2-3p+1)}6 z^3+ \frac {p(-6p^3+12p^2-7p+1)}{24}z^4+  \mathcal{O}(z^5).
$$
\begin{lemma}\label{taylor}
We have
\begin{equation}\label{f}
F(p, z)=  p \sum_{k=1}^{\infty} \frac{P_k(p)}{k!} z^k
\end{equation}
where  $P_k(p)$ is a power series in $p$ with constant term $P_k(0)=1$ for all $k\in\N$.
\end{lemma}

\begin{proof}
Since $F(0,z)=F(p,0)=0$, we can extract $p$ and write the expansion in the form~\eqref{f}. To see that $P_k(0)=1$, we need to check that
$$
\left.\frac {\dd^k}{\dd z^k} \frac {\dd}{\dd p} F(p,z) \right|_{(p,z)=(0,0)} = 1,\qquad k\in\N.
$$
But this is trivial because $\frac {\dd}{\dd p} F(p,z)|_{p=0} = \eee^z-1$. In fact, $P_k(p)$ is even a polynomial in $p$, but we will not need this.
\end{proof}

We will  several times use the following corollary of the above lemma.
\begin{corollary}\label{cor:cosh_subtract}
For $(p,z)\in \mathcal D$ we have
\begin{align}
\frac{F(p,z) + F(p,-z)}{2}
&=
p(\cosh (z) - 1) - \frac {p^2z^2}{2} + p^2z^4 Q(p,z),\label{cor:1}\\
\frac{F(p,z) - F(p,-z)}{2}
&=
pz + pz^3 \tilde Q(p,z),\label{cor:2}
\end{align}
where $Q(p,z)$ and $\tilde Q(p,z)$ are power series of two variables representing analytic functions on $\mathcal D$.
\end{corollary}
\begin{proof}
Both parts follow from Lemma~\ref{taylor}. To prove~\eqref{cor:1}, note that when adding $F(p,z)$ and $F(p,-z)$, all terms with odd powers of $z$ cancel, namely
$$
\frac{F(p,z)+F(p,-z)}{2} = p\sum_{n=1}^\infty \frac{1}{(2n)!} P_{2n}(p) z^{2n}.
$$
Recalling the Taylor series of $\cosh(z)-1 = \sum_{n=1}^\infty \frac{z^{2n}}{(2n)!}$, we can write
$$
\frac{F(p,z) + F(p,-z)}{2} - p(\cosh (z) - 1) = \sum_{n=1}^\infty \frac{p}{(2n)!} (P_{2n}(p)-1) z^{2n}.
$$
The term corresponding to  $n=1$ is $-p^2 z^2/2$, whereas all other terms contain the factor $p^2 z^4$ because the term $P_{2n}(0)=1$ cancels. This proves~\eqref{cor:1}.

To prove~\eqref{cor:2} note that when subtracting $F(p,z)$ and $F(p,-z)$, all terms with even powers of $z$ cancel, namely
$$
\frac{F(p,z)-F(p,-z)}{2} = \sum_{n=1}^\infty \frac{p}{(2n-1)!} P_{2n-1}(p) z^{2n-1}.
$$
The term corresponding to $n=1$ is $pz$, while all other terms contain the factor $pz^3$.
\end{proof}

\section{Expectation and variance}
Let us first define some quantities that will appear in the sequel.
Fix some $\beta\in (0,1)$ once and for all.  Set
$$
\gamma = \gamma_N := \frac{\beta}{2Np}.
$$
Recall that we assume  $pN\to\infty$ and hence $\gamma_N\to 0$ as $N\to\infty$.
Corollary~\ref{cor:cosh_subtract} enables us to compute the expectation and the covariances of $\eee^{-\beta H(\sigma)}$ asymptotically.

\begin{lemma}\label{EexpH}
For all $p=p(N)$ such that $pN \to \infty$ and all $\sigma\in \{-1, +1\}^N$ we have
\begin{multline*}
\E \eee^{-\beta H(\sigma)}=
\exp\left( -\frac{\beta^2}{8}+N^2p \left(\cosh\left(\frac{\beta}{2Np}\right)-1\right)\right.\\
\left. +\frac{\beta}{2N}|\sigma|^2+ \frac1{N^2p^2}\left(C_{N,1} + C_{N,2}\frac{|\sigma|^2}N\right)\right).
\end{multline*}
Here, $(C_{N,1})_{N\in\N}$ and $(C_{N,2})_{N\in\N}$ are bounded sequences that do not depend on $\sigma$.
\end{lemma}
\begin{remark}\label{remsEexpH1}
For constant $p\in (0,1]$ we obtain the simpler formula
$$
\E \eee^{-\beta H(\sigma)}=
\exp\left( \frac{(1-p)\beta^2}{8p}+\frac{\beta}{2N}|\sigma|^2+\mathcal{O}\left(\frac1{N^2}\right)\left(\frac{|\sigma|^2}N+1\right)\right),
$$
where the $\mathcal{O}$-term is uniform in $\sigma \in \{-1,+1\}^N$.
\end{remark}
\begin{remark}\label{rem:cosh}
We can write the $\cosh$-term as a Taylor series
$$
N^2p \left(\cosh\left(\frac{\beta}{2Np}\right)-1\right)
=
N^2p \sum_{k=1}^\infty \frac 1{(2k)!} \left(\frac{\beta}{2Np}\right)^{2k}
=\frac{\beta^2}{8p} + \frac{\beta^4}{24\cdot 16 N^2 p^3} +\ldots.
$$
Note that if we assume for a moment that $pN^{1+\eps} \to \infty$ for some $\eps>0$, then  the sum on the right-hand side only consists of finitely many term that are not $o(1)$. There is a cascade of new  non-negligible summands entering the sum at
$$
p = \mathrm{const.} N^{-\frac{2k-2}{2k-1}}, \quad  k = 2,3,\ldots.
$$
It is the term corresponding to $k=2$ which is responsible for the fact that our results require the assumption $N^2 p^3 \to \infty$.
\end{remark}

\begin{proof}[Proof of Lemma~\ref{EexpH}]
Recall the notation $\gamma = \beta/(2Np)$. We need to study
$$
\E \eee^{-\beta H(\sigma)}
= \E \left[ \eee^{\gamma \sum_{i,j=1}^N \vep_{i,j} \sigma_i \sigma_j} \right]
= \prod_{i,j=1}^N \E \left[ \eee^{\gamma \vep_{i,j} \sigma_i \sigma_j}\right]
=\prod_{i,j=1}^N\left(1-p + p \eee^{\gamma \sigma_i \sigma_j}\right).
$$
Defining $f(x) = f(x; p,\gamma) = \log (1-p + p\eee^{\gamma x})$, we can write
$$
\E \eee^{-\beta H(\sigma)}
=
\exp\left(\sum_{i,j=1}^N \log (1-p + p\eee^{\gamma \sigma_i \sigma_j})\right)
=\exp\left(\sum_{i,j=1}^N f(\sigma_i \sigma_j)\right).
$$
Observe that the argument of $f$, i.e.~$\sigma_i \sigma_j$, can only take the two values $\pm 1$. For these values we can linearize the function $f$, i.e.\ we can write
$$
f(x) = a_0 + a_1 x, \qquad x \in \{-1,+1\},
$$
where $a_0= a_0(p,\gamma)$ and $a_1=a_1(p,\gamma)$ are given by
\begin{align*}
a_0
&= \frac  {f(1) + f(-1)}{2} =\frac{F(p,\gamma)+F(p,-\gamma)}{2}
,\\
a_1
&=
\frac {f(1)-f(-1)}2=\frac{F(p,\gamma)-F(p,-\gamma)}{2},
\end{align*}
and we recall that $F(p,z)=\log (1-p+p\eee^{z})$.
From here we obtain
$$
\E  \eee^{-\beta H(\sigma)}  =
\exp\left(\sum_{i,j=1}^N f(\sigma_i \sigma_j)\right)=
\exp\left(\sum_{i,j=1}^N (a_0+a_1\sigma_i \sigma_j)\right)
=
\exp\left(N^2 a_0+a_1|\sigma|^2\right).
$$

\vspace*{2mm}
\noindent
\textit{Step 1.}
Let us first consider the term $N^2a_0$.
To compute the expansion of $a_0$ we use Corollary~\ref{cor:cosh_subtract} with $(p,z)=(p,\gamma)$:
$$
a_0 = p(\cosh \gamma-1) -\frac{\gamma^2 p^2}2+p^2\gamma^4 Q_1(p,\gamma),
$$
where $Q_1(p,\gamma)$ is an analytic function that is uniformly bounded by some constant $C$ over the region $|p|\leq 1$, $|\gamma|\leq z_0/2$. Note that the latter condition is satisfied for $N$ large enough since $\gamma= \gamma_N\to 0$ as $N\to\infty$. We set  $C_{N,1}: = \beta^4 Q_1(p,\gamma)/16$.
 Thus,
\begin{equation*}
N^2 a_0
=  N^2p\left(\cosh\left(\frac {\beta}{2Np}\right)-1\right) -\frac{\beta^2}8 +  \frac {C_{N,1}}{N^2p^2}
\end{equation*}
and, of course, the constant $C_{N,1}$ does not depend on $\sigma$ and satisfies $|C_{N,1}|\leq C$.

\vspace*{2mm}
\noindent
\textit{Step 2.}
Let us now turn to the term $a_1|\sigma|^2$. By Corollary~\ref{cor:cosh_subtract} we have
$$
a_1= p\gamma + p\gamma^3 Q_2(p,\gamma),
$$
where $Q_2(p,\gamma)$ is an analytic function that is uniformly bounded by some constant $C$ over the region $|p|\leq 1$, $|\gamma|\leq z_0/2$.
Thus,
\begin{eqnarray*}
a_1|\sigma|^2 = \frac{\beta}{2N} |\sigma|^2+ \frac {|\sigma|^2}{N} \frac {C_{N,2}} {N^2 p^2},
\end{eqnarray*}
where $C_{N,2}:=\beta^3 Q_2(p,\gamma)/8$ is again bounded by $C$ in absolute value.
Taking everything together, we obtain the required statement.
\end{proof}

The following lemma is a crucial ingredient in the proof of Theorem \ref{VarZ2}. It is proved using similar ideas as in the proof of Lemma \ref{EexpH}.

\begin{lemma}\label{EexpH1H2}
For $p=p(N)$ such that $pN \to \infty$ and any $\sigma, \tau \in \{-1,+1\}^N$ write $$|\sigma \tau|:=\sum_{i=1}^N \sigma_i \tau_i.$$ Then we have
\begin{multline*}
\E\left[ \eee^{-\beta H(\sigma)} \eee^{-\beta H(\tau)}\right]
=
\exp\left(
\frac{|\sigma|^2+|\tau|^2}N \left(\frac \beta 2 + \frac{C_{N,4}}{N^2p^2} \right) \right.\\
\left.+ \left(\frac{N^2p}2 \left(\cosh\left(\frac{\beta}{Np}\right)-1\right) -\frac{\beta^2}{4} + \frac{C_{N,3}}{N^2p^2}\right) \left(1+\frac{|\sigma\tau|^2}{N^2}\right)
\right).
\end{multline*}
Here, the sequences $(C_{N,3})_{N\in\N}$ and $(C_{N,4})_{N\in\N}$ do not depend on $\sigma, \tau \in \{-1, +1\}^N$ and stay bounded.
\end{lemma}

\begin{remark}\label{remsEexpH1H2}
For constant $p$ we obtain the simpler formula
\begin{multline*}
\E\left[ \eee^{-\beta H(\sigma)} \eee^{-\beta H(\tau)}\right]
= \\
\exp\left( \frac{(1-p)\beta^2}{4p}+\frac{\beta(|\sigma|^2+|\tau|^2)}{2N}+\frac{\beta^2|\sigma \tau|^2}{4N^2}\left(\frac 1 p -1\right)+\mathcal{O}\left(\frac1{N^2}\right)\left(\frac{|\sigma|^2+|\tau|^2}N+1\right)\right),
\end{multline*}
where the $\mathcal{O}$-term is uniform in $\sigma, \tau \in \{-1,+1\}^N$.
\end{remark}

\begin{remark}
We can expand the $\cosh$-term into a Taylor series
$$
\frac{N^2p}2 \left(\cosh\left(\frac{\beta}{Np}\right)-1\right)
=
\frac{N^2p}{2} \sum_{k=1}^\infty \frac 1{(2k)!} \left(\frac{\beta}{Np}\right)^{2k}
=
\frac{\beta^2}{4p} + \frac{\beta^4}{48 N^2 p^3} +\ldots.
$$
Again there is a cascade of new non-negligible terms appearing in the sum at $p=\mathrm{const.} N^{-\frac{2k-2}{2k-1}}$.
\end{remark}

\begin{proof}[Proof of Lemma~\ref{EexpH1H2}]
The principal idea of the proof is similar to the proof of Lemma \ref{EexpH}.

Observe that along the lines of this proof we obtain
\begin{eqnarray*}
\E\left[ \eee^{-\beta H(\sigma)} \eee^{-\beta H(\tau)}\right]&=&\E \left[ \eee^{\gamma \sum_{i,j=1}^N (\sigma_i \sigma_j+\tau_i \tau_j) \vep_{i,j}} \right]
= \E\exp\left(\sum_{i,j=1}^N f(\sigma_i \sigma_j+ \tau_i \tau_j)\right)
\end{eqnarray*}
with the same definition $f(x) = f(x; p,\gamma) = \log (1-p + p\eee^{\gamma x})$.
Note that $f(\sigma_i \sigma_j+\tau_i \tau_j)$ is a function of the arguments $x_1:=\sigma_i\sigma_j$ and $x_2:= \tau_i\tau_j$ which take values in  $\{-1,+1\}$ only. Any such function can be represented in the form
$$
f(x_1+x_2)= b_0 + b_1 x_1 + b_2 x_2 + b_{12} x_1 x_2,
\quad x_1,x_2\in \{-1,+1\},
$$
for suitable coefficients $b_0,b_1,b_2,b_{12}$ depending on $p$ and $\gamma$. This brings us to a system of four linear equations in the four coefficients:
\begin{align*}
b_0 + b_1+b_2 +b_{12} =& f(2),\\
b_0 + b_1-b_2 -b_{12} =& f(0)=0,\\
b_0 - b_1+b_2 -b_{12} =& f(0)=0,\\
b_0 - b_1-b_2 +b_{12} =& f(-2),
\end{align*}
which is solved by
\begin{align*}
b_0
&= b_{12} = \frac {f(2)+f(-2)}4 =\frac {F(p,2\gamma) + F(p,-2\gamma)}4 , \\
b_1
&= b_{2} = \frac {f(2)-f(-2)}4 =\frac {F(p,2\gamma) - F(p,-2\gamma)}4,
\end{align*}
where we recall that $F(p,z)=\log (1-p+p\eee^{z})$.
Therefore we arrive at
$$
\E\left[ \eee^{-\beta H(\sigma)} \eee^{-\beta H(\tau)}\right]= \exp\left\{N^2 b_0 + b_1 |\sigma|^2+b_2 |\tau|^2+b_{12} |\sigma \tau|^2\right\}.
$$
Next, we approximate $b_0,b_1,b_2,b_{12}$. This is again very similar to the proof of Lemma~\ref{EexpH} with the only difference that we now evaluate $F(p,z)$ in the points $z=2\gamma$ and $z=-2\gamma$. In doing so by means of Corollary~\ref{cor:cosh_subtract}, we see that
$$
b_0 = b_{12} =  \frac{p} 2 (\cosh (2\gamma) - 1)  - p^2\gamma^2  + p^2 \gamma^4 Q_3(p,\gamma),
$$
where $Q_3(p,\gamma)$ is an analytic function that is uniformly bounded by some constant $C$ over the region $|p|\leq 1$, $|\gamma|\leq z_0/2$.
From this we obtain
\begin{eqnarray*}
N^2b_0 +|\sigma \tau|^2b_{12}  = \left(\frac{N^2p}2 \left(\cosh\left(\frac{\beta}{Np}\right)-1\right) -\frac{\beta^2}{4}+
\frac {C_{N,3}}{N^2p^2}\right) \left(1+\frac{|\sigma\tau|^2}{N^2}\right),
\end{eqnarray*}
where $C_{N,3}:=\beta^4 Q_3(p,\gamma)/16$.
Now we take a closer look at $b_1 |\sigma|^2+b_2 |\tau|^2$. By Corollary~\ref{cor:cosh_subtract}, we have
$$
b_1 = b_2 = p\gamma + p\gamma^3 Q_4(p,\gamma),
$$
where $Q_4(p,\gamma)$ is an analytic function that is uniformly bounded by some constant $C$ over the region $|p|\leq 1$, $|\gamma|\leq z_0/2$.  With $C_{N,4}:=\beta^3 Q_4(p,\gamma)/8$
this yields
$$
b_1 |\sigma|^2+b_2 |\tau|^2
=\frac{\beta}{2N}(|\sigma|^2+|\tau|^2) + \frac {C_{N,4}} {N^2p^2} \frac{|\sigma|^2+|\tau|^2}N.
$$
The constants $C_{N,3}$ and $C_{N,4}$  do not depend on $\sigma$ and $\tau$ and are uniformly bounded in $N$. Putting these observations together gives the required statement.
\end{proof}

\section{Proofs of Theorems \ref{theoCW1}, \ref{EZNg}, and \ref{VarZ2}}
Now we are able to prove the theorems stated in Section 1. We start with Theorem~\ref{EZNg}.
\begin{proof}[Proof of Theorem  \ref{EZNg}]
To shorten the notation, let us write
$$
A_{N}(\beta) := -\frac{\beta^2}{8}+N^2p \left(\cosh\left(\frac{\beta}{2Np}\right)-1\right).
$$
Our aim is to prove that
\begin{equation}\label{eq:Z_N_beta_g_proof}
\E Z_N(\beta, g)
\sim
2^N \eee^{A_N(\beta)} \E_{\xi} [g(\xi)\eee^{\frac {\beta}{2}\xi^2}].
\end{equation}
The main idea is to divide the set of spin configurations into ``typical'' and ``atypical'' configurations. To formalize this idea, we denote the set of ``typical'' configurations by
 \begin{equation*}
 T_N := \{\sigma\in \{-1,+1\}^N\colon |\sigma|^2 \leq N^2 p\}.
 \end{equation*}
The configurations in $T_N$ are ``typical'' in the sense that $\sigma$ picked from $\{-1,+1\}^N$ uniformly at random belongs to $T_N$ with probability converging to $1$. This follows from the de Moivre-Laplace central limit theorem. Let $T_N^c:= \{-1,+1\}^N\backslash T_N$ be the set of atypical spin configurations.    Evidently,
\begin{equation}\label{eq:expect_typical_atypical}
\E Z_N(\beta, g)
=
\sum_{\sigma \in T_N} g\left( \frac{|\sigma|}{\sqrt N} \right) \E \eee^{-\beta H(\sigma)}
+
\sum_{\sigma \in T_N^c} g\left( \frac{|\sigma|}{\sqrt N} \right) \E \eee^{-\beta H(\sigma)}.
\end{equation}
Recall the result of Lemma~\ref{EexpH}: For some bounded sequences $(C_{N,1})_{N\in\N}$ and $(C_{N,2})_{N\in\N}$, we have
\begin{equation}\label{eq:expect_e_beta_H_repeat}
\E \eee^{-\beta H(\sigma)}=
\exp\left( A_N(\beta) + \frac{\beta}{2N}|\sigma|^2+ \frac1{N^2p^2}\left(C_{N,1} + C_{N,2}\frac{|\sigma|^2}N\right)\right).
\end{equation}
In the proof we consider the two sums in \eqref{eq:expect_typical_atypical} separately.
We start with the second term.

\vspace*{2mm}
\noindent
\textit{Step 1: Atypical $\sigma$'s.}
We claim that
$$
\sum_{\sigma \in T_N^c} g\left( \frac{|\sigma|}{\sqrt N} \right) \E \eee^{-\beta H(\sigma)}
=  o\left(2^N \eee^{A_N(\beta)}\right).
$$
With $\|g\|_\infty:= \sup_{t\in\R} |g(t)| < \infty$ we have
\begin{align*}
\left|\sum_{\sigma \in T_N^c} g\left( \frac{|\sigma|}{\sqrt N} \right) \E \eee^{-\beta H(\sigma)}\right|
&\leq
\|g\|_\infty \sum_{\sigma \in T_N^c}\E \eee^{-\beta H(\sigma)}\\
&\leq
\|g\|_\infty \sum_{\sigma \in T_N^c} \eee^{A_N(\beta) + \frac{\beta}{2N}|\sigma|^2+ \frac1{N^2p^2}\left(C_{N,1} + C_{N,2}\frac{|\sigma|^2}N\right)}\\
&=
\|g\|_\infty \eee^{A_N(\beta)} \sum_{k\in\Z : |k|>N\sqrt p}
\eee^{\frac{\beta}{2N} k^2 + \frac1{N^2p^2}\left(C_{N,1} + C_{N,2}\frac{k^2}N\right)} \nu_N(k),
\end{align*}
where $\nu_N(k)$ is the number of $\sigma\in \{-1,+1\}^N$ such that $|\sigma| = k$. By the Local Limit Theorem, we have
$$
2^{-N}\nu_N(k) \leq \frac{C}{\sqrt N} \eee^{-\frac{k^2}{2N}}, \quad k\in\Z,
$$
where $C$ is an absolute constant. Using this estimate, we obtain
\begin{align*}
\left|\sum_{\sigma \in T_N^c} g\left( \frac{|\sigma|}{\sqrt N} \right) \E \eee^{-\beta H(\sigma)}\right|
&\leq \|g\|_\infty
C \eee^{A_N(\beta)} \frac {2^N} {\sqrt N} \sum_{k\in\Z : |k|>N\sqrt p}
\eee^{\frac{k^2}{2N} \left(\beta-1 + \frac{2C_{N,1}+2C_{N,2}}{N^2p^2}\right)},
\end{align*}
for sufficiently large $N$, where we used that $k^2/N\geq N^2p/N =Np\to +\infty$.
Since $\beta<1$, and $C_{N,1}$ and $C_{N,2}$ are bounded, there is $\eps>0$ such that
$$
\beta-1 + \frac{2C_{N,1}+2C_{N,2}}{N^2p^2} < -\eps
$$
for all sufficiently large $N$. It follows that
\begin{align*}
\left|\sum_{\sigma \in T_N^c} g\left( \frac{|\sigma|}{\sqrt N} \right) \E \eee^{-\beta H(\sigma)}\right|
&\leq \|g\|_\infty
C \eee^{A_N(\beta)} \frac {2^N} {\sqrt N} \sum_{k\in\Z : |k|>N\sqrt p}
\eee^{-\eps \frac{k^2}{2N}}
\\
&\leq \|g\|_\infty
C \eee^{A_N(\beta)} 2^N \int_{\sqrt{Np}/2}^{\infty} \eee^{-\eps t^2/2} \dd  t,
\end{align*}
by enlargening the range of integration.
This proves the claim since $\sqrt{Np} \to +\infty$ as $N\to\infty$.

\vspace*{2mm}
\noindent
\textit{Step 2: Typical $\sigma$'s.}
For $\sigma\in T_N$ we can simplify the result of Lemma~\ref{EexpH} as follows:
\begin{eqnarray*}
\E \eee^{-\beta H(\sigma)}&=&
\exp\left( A_N(\beta)+\frac{\beta}{2N}|\sigma|^2+\frac{C_{N,5}(\sigma)}{Np}\right),
\end{eqnarray*}
where $C_{N,5}(\sigma)$ is uniformly bounded over all $N$ and all typical $\sigma\in T_N$.
Using this estimate we obtain
\begin{align*}
\sum_{\sigma \in T_N} g\left( \frac{|\sigma|}{\sqrt N} \right) \E \eee^{-\beta H(\sigma)}
&\sim \eee^{A_N(\beta)} \sum_{\sigma \in T_N} g\left( \frac{|\sigma|}{\sqrt N} \right)  \eee^{\frac{\beta}{2N}|\sigma|^2}\\
&\sim \eee^{A_N(\beta)} \sum_{\sigma \in \{-1,+1\}^N} g\left( \frac{|\sigma|}{\sqrt N} \right)  \eee^{\frac{\beta}{2N}|\sigma|^2},
\end{align*}
where the second asymptotic equivalence holds because
$$
\sum_{\sigma \in T_N^c} g\left( \frac{|\sigma|}{\sqrt N} \right)  \eee^{\frac{\beta}{2N}|\sigma|^2} = o\left(\sum_{\sigma \in \{-1,+1\}^N} g\left( \frac{|\sigma|}{\sqrt N} \right)  \eee^{\frac{\beta}{2N}|\sigma|^2}\right).
$$
To see that this is true, note that, as we will show in Step 3 below,
$$
\sum_{\sigma \in \{-1,+1\}^N} g\left( \frac{|\sigma|}{\sqrt N} \right)  \eee^{\frac{\beta}{2N}|\sigma|^2} \sim 2^N \E_{\xi}[g(\xi)\eee^{\frac {\beta}{2}\xi^2}],
$$
with $\E_{\xi}[g(\xi)\eee^{\frac {\beta}{2}\xi^2}] > 0$ (since $g$ is continuous, non-negative and $g\not \equiv 0$), and
$$
\sum_{\sigma \in T_N^c} g\left( \frac{|\sigma|}{\sqrt N} \right)  \eee^{\frac{\beta}{2N}|\sigma|^2} = o(2^N)
$$
by an argument similar to that used in Step 1.

\vspace*{2mm}
\noindent
\textit{Step 3.} It remains to show that
\begin{equation}\label{eq:sum_g_exp_de_Moivre}
\lim_{N\to\infty} \frac 1 {2^N} \sum_{\sigma \in \{-1,+1\}^N} g\left( \frac{|\sigma|}{\sqrt N} \right)  \eee^{\frac{\beta}{2N}|\sigma|^2}
= \E_{\xi}[g(\xi)\eee^{\frac {\beta}{2}\xi^2}],
\end{equation}
where $\xi$ is a standard normal random variable.  Let us consider $\sigma$ as a random element in $\{-1,+1\}^N$ sampled according to the uniform probability  distribution assigning the same probability $2^{-N}$ to each element of $\{-1,+1\}^N$.  By the Central Limit Theorem of de Moivre--Laplace,
$$
\frac{|\sigma|}{\sqrt N} \todistr \xi.
$$
The idea is to use the continuous mapping theorem to prove~\eqref{eq:sum_g_exp_de_Moivre}.  Unfortunately, the function $z\mapsto g(z) \eee^{\frac \beta 2 z^2}$ is not bounded, which means that some work needs to be done to obtain~\eqref{eq:sum_g_exp_de_Moivre}. However, since this function is continuous, the continuous mapping theorem tells us that
\begin{equation}\label{eq:sum_g_exp_de_Moivre_aux1}
W_N := g\left( \frac{|\sigma|}{\sqrt N} \right)  \eee^{\frac{\beta}{2N}|\sigma|^2} \todistr g(\xi)\eee^{\frac {\beta}{2}\xi^2}.
\end{equation}
As shown in~\cite{Ellis-EntropyLargeDeviationsAndStatisticalMechanics}, Proof of Theorem V.9.4, for $\beta<1$ the random variables $W_N$ satisfy
\begin{equation}\label{eq:sum_g_exp_de_Moivre_aux2}
\sup_{N\in\N} \P_\sigma [|W_N| \geq  u] \leq C u^{-1/\beta} \text{ for all } u>0.
\end{equation}
In fact, \cite{Ellis-EntropyLargeDeviationsAndStatisticalMechanics}, Proof of Theorem V.9.4 only proves this for the $g=1$ (and then obtains $C=2$ for the constant), but since $g$ is bounded, the estimate still holds with some $C=C(g)$. By uniform integrability it follows from~\eqref{eq:sum_g_exp_de_Moivre_aux1} and~\eqref{eq:sum_g_exp_de_Moivre_aux2} that
$$
\lim_{N\to\infty} \E_{\sigma} [W_N] = \E_{\xi} [g(\xi)\eee^{\frac {\beta}{2}\xi^2}],
$$
which completes the proof of~\eqref{eq:sum_g_exp_de_Moivre} and thus the proof of part (a) of Theorem~\ref{EZNg}.

Statement (b) of Theorem~\ref{EZNg} is immediate and statment (c) follows by Gaussian integration.
\end{proof}

We now continue by computing the variance of $Z_N(\beta, g)$.
\begin{proof}[Proof of Theorem \ref{VarZ2}]
The key observation, and the reason, why we need the condition $N^2p^3\to \infty$, is that under this condition
\begin{equation}\label{eq:A_N_beta_asympt}
2A_{N}(\beta) := -\frac{\beta^2}{4}+ 2N^2p \left(\cosh\left(\frac{\beta}{2Np}\right)-1\right) = -\frac{\beta^2}{4}+\frac{\beta^2}{4p}+o(1),
\end{equation}
as well as
\begin{equation}\label{eq:B_N_beta_asympt}
B_{N}(\beta) := -\frac{\beta^2}{4}+\frac{N^2p}{2} \left(\cosh\left(\frac{\beta}{Np}\right)-1\right)
=
-\frac{\beta^2}{4}+\frac{\beta^2}{4p}+o(1).
\end{equation}
Using the shorthands $A_N(\beta)$ and $B_N(\beta)$ we can write the results of Lemmas~\ref{EexpH} and~\ref{EexpH1H2} as follows:
\begin{align}
\E [\eee^{-\beta H(\sigma)}] \E [\eee^{-\beta H(\tau)}]
&= 
\eee^{2A_N(\beta) + \frac{\beta}{2} \frac {|\sigma|^2+|\tau|^2} N + \frac 1 {N^2p^2}   \left(2C_{N,1} + C_{N,2}\frac {|\sigma|^2 + |\tau|^2}N\right)}, \label{eq:cov_computation1}\\
\E [\eee^{-\beta H(\sigma)} \eee^{-\beta H(\tau)}]
&=
\eee^{
\left(B_N(\beta) + \frac{C_{N,3}}{N^2p^2}\right) \left(1+\frac{|\sigma\tau|^2}{N^2}\right)
+
\frac{|\sigma|^2+|\tau|^2}N \left(\frac \beta 2 + \frac{C_{N,4}}{N^2p^2} \right)
}.\label{eq:cov_computation2}
\end{align}
Again, we consider ``typical'' pairs of configurations, which in this case, by definition,  lie  in the set
\begin{equation}\label{eq:S_N_def}
S_N:= \{ (\sigma,\tau) \in \{-1,+1\}^N \times \{-1,+1\}^N: |\sigma|^2 \leq N^2 p, \quad |\tau|^2 \leq N^2 p, \quad |\sigma \tau|^2 \leq N^{4/3}\}.
\end{equation}
 We decompose the sums occurring in the definition of the variance into sums over typical and atypical pairs $(\sigma, \tau)$:
\begin{multline*}
\V(Z_N(\beta,g))
= \sum_{(\sigma, \tau) \in S_N} g\left(\frac{|\sigma|}{\sqrt N}\right) g\left(\frac{|\tau|}{\sqrt N}\right) \Co\left(\eee^{-\beta H(\sigma)}, \eee^{-\beta H(\tau)}\right)\\
+ \sum_{(\sigma, \tau) \in S_N^c}g\left(\frac{|\sigma|}{\sqrt N}\right)g\left(\frac{|\tau|}{\sqrt N}\right) \Co\left(\eee^{-\beta H(\sigma)}, \eee^{-\beta H(\tau)}\right).
\end{multline*}

\vspace*{2mm}
\noindent
\textit{Step 1: Typical pairs $(\sigma, \tau)$}.
For $(\sigma, \tau) \in S_N$ straightforward estimates together with the asymptotic behaviour of $A_N(\beta)$ and $B_N(\beta)$ given in \eqref{eq:A_N_beta_asympt} and \eqref{eq:B_N_beta_asympt} show that formulae~\eqref{eq:cov_computation1} and~\eqref{eq:cov_computation2} simplify to
\begin{align*}
\E [\eee^{-\beta H(\sigma)}] \E [\eee^{-\beta H(\tau)}]
&=
\exp\left(-\frac{\beta^2}{4}+\frac{\beta^2}{4p} + \frac \beta 2 \frac{|\sigma|^2 + |\tau|^2}{N} + \eps_N'(\sigma,\tau)\right),\\
\E [\eee^{-\beta H(\sigma)} \eee^{-\beta H(\tau)}]
&=
\exp\left(-\frac{\beta^2}{4}+\frac{\beta^2}{4p} + \frac \beta 2 \frac{|\sigma|^2 + |\tau|^2}{N} + \eps_N''(\sigma,\tau)\right),
\end{align*}
where $\eps_N'(\sigma,\tau)$ and $\eps_N''(\sigma,\tau)$ satisfy
$$
\lim_{N\to\infty} \max_{(\sigma,\tau)\in S_N} |\eps_N'(\sigma,\tau)|
=
\lim_{N\to\infty} \max_{(\sigma,\tau)\in S_N} |\eps_N''(\sigma,\tau)|
=
0.
$$
Recalling that the function $g$ is bounded, we obtain
\begin{multline*}
\sum_{(\sigma, \tau) \in S_N} g\left(\frac{|\sigma|}{\sqrt N}\right) g\left(\frac{|\tau|}{\sqrt N}\right) \Co\left(\eee^{-\beta H(\sigma)}, \eee^{-\beta H(\tau)}\right)
\\=
o(1)\cdot
\sum_{(\sigma, \tau) \in S_N}
\exp\left(-\frac{\beta^2}{4}+\frac{\beta^2}{4p} + \frac \beta 2 \frac{|\sigma|^2 + |\tau|^2}{N}\right).
\end{multline*}
Next we claim that
\begin{equation}
\label{eq:claim_sum_sigma_tau}
\sum_{(\sigma, \tau) \in S_N}
\exp\left(-\frac{\beta^2}{4}+\frac{\beta^2}{4p} + \frac \beta 2 \frac{|\sigma|^2 + |\tau|^2}{N}\right) \leq C (\E Z_N(\beta, g))^2,
\end{equation}
for some constant $C$  (which may depend on $g$).
We have
\begin{align}
&\sum_{(\sigma, \tau) \in S_N}
\exp\left(-\frac{\beta^2}{4}+\frac{\beta^2}{4p} + \frac \beta 2 \frac{|\sigma|^2 + |\tau|^2}{N}\right)
\leq \eee^{-\frac{\beta^2}{4}}\eee^{\frac{\beta^2}{4p}}\sum_{(\sigma, \tau)\in S_N}
\exp\left( \frac \beta 2 \frac{|\sigma|^2 + |\tau|^2}{N}\right)
\nonumber
\\
& \leq c
\eee^{-\frac{\beta^2}{4}}\eee^{\frac{\beta^2}{4p}}(2^N)^2 \left( \E_{\xi} [\eee^{\frac{\beta}{2}\xi^2}]\right)^2,
\label{ineq:sum_sigma_tau}
\end{align}
for some constant $c$, where we used \eqref{eq:sum_g_exp_de_Moivre} (for  $g\equiv 1$).
By Theorem~\ref{EZNg} (b) we have
\begin{equation}\label{eq:asymp_Z_N^2}
(2^N)^2 \left( \mathbb E_{\xi} [\eee^{\frac{\beta}{2}\xi^2}]\right)^2 \sim \left(\E Z_N (\beta,1)\right)^2\eee^{-\frac{(1-p)\beta^2}{4p}} \sim C' \left(\E Z_N (\beta,g)\right)^2\eee^{-\frac{(1-p)\beta^2}{4p}}
\end{equation}
for some constant $C'=C_g'$. Inserting \eqref{eq:asymp_Z_N^2} into \eqref{ineq:sum_sigma_tau} proves the claim in \eqref{eq:claim_sum_sigma_tau}.
Hence we have
$$
\sum_{(\sigma, \tau) \in S_N} g\left(\frac{|\sigma|}{\sqrt N}\right) g\left(\frac{|\tau|}{\sqrt N}\right) \Co\left(\eee^{-\beta H(\sigma)}, \eee^{-\beta H(\tau)}\right)
=
o(1)\cdot (\E Z_N(\beta, g))^2.
$$
\vspace*{2mm}
\noindent
\textit{Step 2: Atypical pairs $(\sigma, \tau)$}. Let $V_N(k,l,m)$ be the set of pairs
$$(\sigma,\tau)\in \{-1,+1\}^N\times\{-1,+1\}^N \quad \mbox{for which }|\sigma|=k,|\tau|= l, \mbox{and }|\sigma \tau| = m.$$

Denote by $\nu_N(k,l,m) = \# V_N(k,l,m)$ the number of such pairs. If we sample $\sigma=(\sigma_1,\ldots,\sigma_N)$ and $\tau=(\tau_1,\ldots,\tau_N)$ independently and uniformly from $\{-1,+1\}^N$, then we can regard $(\sigma_i,\tau_i,\sigma_i\tau_i)$, $1\leq i\leq N$, as i.i.d.\ three-dimensional random vectors with zero mean.
The covariance matrix of these random vectors is the $3\times 3$ identity matrix because
$$
\sigma_i (\sigma_i\tau_i) = \tau_i, \quad \tau_i (\sigma_i\tau_i) = \sigma_i, \quad \sigma_i^2 = \tau_i^2 = (\sigma_i\tau_i)^2 = 1.
$$
  By the three-dimensional Local Central Limit Theorem~\cite{Davis1995}, there is a universal constant $C$ such that
$$
\nu_N(k,l,m) \leq C 2^{2N} N^{-3/2} \eee^{-\frac{k^2}{2N} - \frac{l^2}{2N} - \frac{m^2}{2N}}, \quad (k,l,m)\in\Z^3.
$$
It follows from this and~\eqref{eq:cov_computation2} that for every $(k,l,m)\in\Z^3$,
\begin{multline*}
2^{-2N}\sum_{(\sigma,\tau) \in V_N(k,l,m)}
\E [\eee^{-\beta H(\sigma)} \eee^{-\beta H(\tau)}]
\leq \\
CN^{-3/2}
\exp\left(
\left(B_N(\beta) + \frac{C_{N,3}}{N^2p^2}\right) \left(1+\frac{m^2}{N^2}\right) - \frac{m^2}{2N}
+
\frac{k^2+l^2}N \left(\frac {\beta-1} 2 + \frac{C_{N,4}}{N^2p^2} \right)
\right).
\end{multline*}
Note that $B_N(\beta) = 2A_N(\beta)+o(1) = O(1/p)$ by~\eqref{eq:A_N_beta_asympt}, \eqref{eq:B_N_beta_asympt} and hence
$$
\left(B_N(\beta) + \frac{C_{N,3}}{N^2p^2}\right) \left(1+\frac{m^2}{N^2}\right) = 2 A_N(\beta) +o(1) + \frac{m^2}{N^2} O(1/p) \leq 2 A_N(\beta) +o(1) + \frac{m^2}{4N}.
$$
Since $\beta<1$, we obtain for some $\varepsilon >0$ and $N$ sufficiently large
\begin{equation*}
2^{-2N}\sum_{(\sigma,\tau) \in V_N(k,l,m)}
\E [\eee^{-\beta H(\sigma)} \eee^{-\beta H(\tau)}]
\leq
CN^{-3/2}
\eee^{2A_N(\beta)}
\eee^{-\eps \frac{k^2+l^2 + m^2} {2N}}.
\end{equation*}

A similar argument applies to the sum of $\E [\eee^{-\beta H(\sigma)}] \E[\eee^{-\beta H(\tau)}]$, i.e.\ with \eqref{eq:cov_computation1} and again the bound on $\nu_N(k,l,m) $, we have
\begin{align*}
&2^{-2N}\sum_{(\sigma,\tau) \in V_N(k,l,m)}
\E [\eee^{-\beta H(\sigma)}] \E[ \eee^{-\beta H(\tau)}]
\\&\leq C   N^{-3/2} \eee^{-\frac{k^2}{2N} - \frac{l^2}{2N} - \frac{m^2}{2N}} \eee^{2A_N(\beta) + \frac{\beta}{2} \frac {k^2+l^2} N + \frac 1 {N^2p^2}   \left(2C_{N,1} + C_{N,2}\frac {k^2 + l^2}N\right)}
\\
&\leq CN^{-3/2}
\eee^{2A_N(\beta)}
\eee^{-\eps \frac{k^2+l^2 + m^2} {2N}}
\end{align*}
for some $\varepsilon >0$ and $N$ sufficiently large.

Thus we have
$$
2^{-2N}\sum_{(\sigma,\tau) \in V_N(k,l,m)}
\left| \Cov (\eee^{-\beta H(\sigma)}, \eee^{-\beta H(\tau)})\right|
\leq
CN^{-3/2}
\eee^{2A_N(\beta)}
\eee^{-\eps \frac{k^2+l^2 + m^2} {2N}}
$$
for some $\varepsilon >0$ and $N$ sufficiently large.
Since the function $g$ is bounded, we obtain
\begin{multline*}
\left|\sum_{(\sigma,\tau) \in S_N^c}
g\left(\frac{|\sigma|}{\sqrt N}\right) g\left(\frac{|\tau|}{\sqrt N}\right)
 \Cov (\eee^{-\beta H(\sigma)}, \eee^{-\beta H(\tau)}) \right|
\\
\leq
C \|g\|_\infty^2 2^{2N} \eee^{2A_N(\beta)}  N^{-3/2}
\sum_{\substack{(k,l,m)\in\Z^3\\N^{-1/2}(k,l,m)\in D_N}} \eee^{-\eps \frac{k^2+l^2 + m^2} {2N}},
\end{multline*}
where $D_N:=\{(x,y,z)\in\R^3\colon |x|>\sqrt {Np} \text{ or } |y|>\sqrt{Np} \text{ or } |z|>N^{1/6}\}$.
Estimating the Riemann sum on the right-hand side by the Riemann integral (over a slightly larger domain), we obtain
$$
 N^{-3/2}
\sum_{\substack{(k,l,m)\in\Z^3\\N^{-1/2}(k,l,m)\in D_N}} \eee^{-\eps \frac{k^2+l^2 + m^2} {2N}}
\leq
\int_{\frac 12 D_N} \eee^{-\eps \frac{x^2+y^2 + z^2} {2}} \dd x\, \dd y \,  \dd z  = o(1)
$$
because $\sqrt {Np} \to \infty$ as $N\to\infty$. In view of~\eqref{eq:Z_N_beta_g_proof} it follows that
$$
\sum_{(\sigma, \tau) \in S_N^c} g\left(\frac{|\sigma|}{\sqrt N}\right) g\left(\frac{|\tau|}{\sqrt N}\right) \Co\left(\eee^{-\beta H(\sigma)}, \eee^{-\beta H(\tau)}\right)
=
o(1)\cdot (\E Z_N(\beta, g))^2.
$$
Combining the results of Step 1 and Step 2 completes the proof of Theorem \ref{VarZ2}.
\end{proof}

We are now ready to prove Theorems \ref{theoCW1} and \ref{theo:magnetization}.

\begin{proof}[Proof of Theorems \ref{theoCW1} and \ref{theo:magnetization}]
Theorem \ref{VarZ2} shows that, if $N^2p^3\to \infty$, we have that $\V(Z_N(\beta,g))= o((\E Z_N(\beta, g))^2)$.
Therefore,
\begin{equation}\label{eq:in_probab}
\frac{ Z_N(\beta, g)}{\E Z_N(\beta, g)} \to 1
\end{equation}
in $L^2$, for all non-negative $g \in \mathcal{C}_b(\R)$, $g\not  \equiv 0$. By Chebyshev's inequality, this convergence holds in probability, too. This proves Theorem~\ref{theoCW1}.

Recall that $L_N$ is the random probability measure on $\R$ defined in~\eqref{eq:def_L_N}.
For all non-negative $g \in \mathcal{C}_b(\R)$  we have
$$
\int_{-\infty}^{+\infty} g(x) L_N(\dd x)
=
\E_{\mu_\beta }\left[ g\left( \frac{\sum_{i=1}^N \sigma_i}{\sqrt N} \right)\right]
=
\frac{Z_N(\beta, g)}{Z_N(\beta)}.
$$
It follows from~\eqref{eq:in_probab} and Theorem~\ref{EZNg} (b) and (c) that
$$
\lim_{N\to\infty}\int_{-\infty}^{+\infty} g(x) L_N(\dd x)
=
\lim_{N\to\infty} \frac{\E Z_N(\beta, g)}{\E Z_N(\beta)}
=
\sqrt{1-\beta}\, \E_\xi [g(\xi) \eee^{\frac \beta 2 \xi^2}]
$$
in probability, where $\xi$ is a standard normally distributed random variable.

As the right hand side equals $\int_{-\infty}^{+\infty} g(x) \phi_{0, \frac 1 {1-\beta}}(x) dx$, where
$\phi_{0, \frac 1 {1-\beta}}(x)$ is the density of a normal distribution with mean $0$ and variance $1/(1-\beta)$, we have shown that $L_N$, considered as a random element of the space of probability measures $\mathcal M(\R)$, converges in probability   to a normal distribution with mean $0$ and variance $1/(1-\beta)$, considered as a deterministic point in $\mathcal M(\R)$.
\end{proof}

\begin{remark}
So far, we do not prove that there is no Central Limit Theorem for $\sqrt N m_N$ when $p^3 N^2$ does not diverge to infinity. However, what is suggested by the above computations is that $\V Z_N(\beta,g)$ and $(\E Z_N(\beta,g))^2$ are of the same order provided that $p^3 N^2 \to c\in (0,\infty)$, hence the behaviour of $Z_N(\beta,g)$ as well as $m_N(\sigma)$ may very well change, if $p^3 N^2$ does not diverge to infinity.

Indeed, in forthcoming work we plan to prove Central Limit Theorem for $m_N(\sigma)$ as well as for $Z_N(\beta)$ for smaller values of $p$. To this end we will have to rely on techniques that are different from techniques used in this note. We hope that they will also allow us to prove a (probably non-standard) Central Limit Theorem for $m_N(\sigma)$ at the critical temperature.
\end{remark}

\bibliographystyle{abbrv}
\bibliography{LiteraturDatenbank}
\end{document}